\documentclass{amsart}
\usepackage{epsfig}
\usepackage{graphicx}
\usepackage{amsfonts}
\usepackage{amssymb}
\usepackage{qtree}
\usepackage{amsmath,latexsym,amsthm}

\newtheorem{theorem}{Theorem}[section]
\newtheorem{Definition}{\bf Definition}[section]
\newtheorem{Thm}[Definition]{\bf Theorem}

\newtheorem{conjecture}[theorem]{Conjecture}

\theoremstyle{definition}
\newtheorem{definition}[theorem]{Definition}

\theoremstyle{remark}

\begin{document}

\title[An interesting family]{A remarkable sequence of integers}

\author{Victor H. Moll}
\address{Department of Mathematics,
Tulane University, New Orleans, LA 70118}
\email{vhm@math.tulane.edu}

\author{Dante V. Manna}
\address{Department of Mathematics,
Virginia Wesleyan College, Norfolk, VA 23502}
\email{dmanna@vwc.edu}


\date{\today}


\begin{abstract}
A survey of properties of a sequence of coefficients appearing in the 
evaluation of a quartic definite integral is presented. These properties 
are of analytical, combinatorial and number-theoretical nature. 

\end{abstract}

\maketitle

\newcommand{\no}{\noindent}
\newcommand{\nn}{\nonumber}
\newcommand{\ba}{\begin{eqnarray}}
\newcommand{\ea}{\end{eqnarray}}
\newcommand{\realpart}{\mathop{\rm Re}\nolimits}
\newcommand{\imagpart}{\mathop{\rm Im}\nolimits}
\newcommand{\stir}{\left[ \begin{smallmatrix}
                         m \\
                         j 
                      \end{smallmatrix}
                      \right]}
\newcommand{\quant}{\genfrac{\langle}{\rangle}{0pt}{}{n}{k}}
\newcommand{\quanuv}{\genfrac{\langle}{\rangle}{0pt}{}{n+mu}{mv}}
\newcommand{\quann}{{\langle { n } \rangle}}
\newcommand{\quank}{{\langle { k } \rangle}}
\newcommand{\quannk}{{\langle {n-k} \rangle}}
\newcommand{\quanone}{{\langle { 1 } \rangle}}
\newcommand{\quantwo}{{\langle { 2 } \rangle}}
\newcommand{\quanthree}{\genfrac{\langle}{\rangle}{0pt}{}{n+u}{v}}
\newcommand{\quantfour}{\genfrac{\langle}{\rangle}{0pt}{}{n+2u}{2v}}

\numberwithin{equation}{section}

\section{A quartic integral} \label{intro} 
\setcounter{equation}{0}

The problem of explicit evaluation of definite integrals has been greatly 
simplified due to the advances in symbolic languages like Mathematica and 
Maple. Some years ago the first author described in \cite{moll-notices} how he 
got interested in these topics and the appearance of the sequence 
of rational numbers
\begin{equation}
d_{l,m} = 2^{-2m} \sum_{k=l}^{m} 2^{k} \binom{2m-2k}{m-k} \binom{m+k}{m} 
\binom{k}{l},
\label{positive-0}
\end{equation}
\noindent
for $0 \leq l \leq m$. These are rational numbers with a simple denominator. 
The numbers $2^{2m}d_{l,m}$ are the remarkable integers in the title. 
These rational coefficients $d_{l,m}$ appeared in the evaluation of the 
{\em quartic integral} 
\begin{equation}
N_{0,4}(a;m) := \int_{0}^{\infty} \frac{dx}{(x^{4} + 2ax^{2} + 1)^{m+1}},
\label{int-deg4}
\end{equation}
\noindent
for $a> -1, \, m \in \mathbb{N}$. The formula
\begin{equation}
N_{0,4}(a;m) = 
\frac{\pi}{2} \frac{P_{m}(a)}{\left[ 2(a+1) \right]^{m + \tfrac{1}{2} } },
\label{int-qua}
\end{equation}
\noindent
with
\begin{equation}
P_{m}(a) = \sum_{l=0}^{m} d_{l,m}a^{l}
\label{polyP-def}
\end{equation}
\noindent
has been established by a variety of methods, some of which are 
 reviewed in \cite{amram}. 
The symbolic status of (\ref{int-deg4}) has not changed much since 
we last reported on \cite{moll-notices}. Mathematica 6.0 is unable to 
compute it when $a$ and $m$ are entered as parameters. On the other hand, the
corresponding indefinite integral is evaluated in terms of the Appell-F1 
function defined by
\begin{equation}
F_{1}(a;b_{1},b_{2};c;x,y) := \sum_{m=0}^{\infty} \sum_{n=0}^{\infty}
\frac{(a)_{m+n} (b_{1})_{m} (b_{2})_{n}}{m! n! (c)_{m+n}} x^{m}y^{n} 
\end{equation}
\noindent
as
\begin{eqnarray}
\int \frac{dx}{(x^{4}+2ax^{2}+1)^{m+1}} & = & 
x \, F_{1} \left[ \frac{1}{2}, 1+m,1+m, \frac{3}{2}, 
- \frac{x^{2}}{a_{+}}, 
\frac{x^{2}}{-a_{-}} \right], \nonumber
\end{eqnarray}
\noindent
where $a_{\pm} := a \pm \sqrt{-1+a^{2}}$.  Here $(a)_{k} = a(a+1) \cdots 
(a+k-1)$ is the ascending factorial. 

The coefficients $\{d_{l,m}: 0 \leq l \leq m \}$ have remarkable properties 
that will be discussed here. Those properties have mainly been discovered
by following the methodology of Experimental Mathematics, 
as presented in \cite{borw1, borw2}. Many of the properties presented 
here have been {\em guessed} using a symbolic language and subsequently 
established by traditional methods. The reader will find in \cite{irrbook}
a detailed introduction to the polynomial $P_{m}(a)$ in (\ref{polyP-def}).

\section{A triple sum expression for $d_{l,m}$} \label{sec-triple} 
\setcounter{equation}{0}

Our first approach to the evaluation of (\ref{int-qua}) was a byproduct
of a new proof of Wallis's formula, 
\begin{eqnarray}
J_{2,m} := \int_{0}^{\infty} \frac{dx}{(x^{2} + 1)^{m+1}} & = & 
\frac{\pi}{2^{2m+1}} \binom{2m}{m}, \label{wallis}
\end{eqnarray}
\noindent
where $m$ is a nonnegative integer. Wallis' formula has the equivalent form 
\begin{equation}
\frac{\pi}{2}  = \frac{2}{1} \cdot \frac{2}{3} \cdot \frac{4}{3} 
\cdot \frac{4}{5} \cdots \frac{2n}{2n-1} \cdot \frac{2n}{2n+1} \cdots.
\end{equation}
\noindent
The reader will find in \cite{irrbook} a proof of the equivalence of these two 
formulations. 

We describe in \cite{sarah1} our first proof of (\ref{wallis}). Section \ref{sec-single}
shows that a 
simple extension leads naturally to the concept of {\em rational Landen 
transformations}. These are transformations on the coefficients of a 
rational integrand that preserve the value of the integral. It is the 
rational analog of the well known transformation 
\begin{equation}
a \mapsto \frac{a+b}{2}, \quad b \mapsto \sqrt{ab}
\end{equation}
\noindent
that preserves the elliptic integral
\begin{equation}
G(a,b) = \int_{0}^{\pi/2} \frac{d \theta}{\sqrt{a^{2} \cos^{2} \theta + 
b^{2} \sin^{2} \theta}}. 
\end{equation}
\noindent
The reader will find in \cite{borwein1} and \cite{manna-moll3} details 
about these topics. 

The proof of Wallis' formula begins with the change 
of variables $x = \tan \theta$. This converts $J_{2,m}$ to its 
trigonometric form 
\begin{eqnarray}
J_{2,m} = \int_{0}^{\pi/2} \cos^{2m} \theta \, d \theta 
& = & \frac{\pi}{2^{2m+1}} \binom{2m}{m}. \label{wallis12}
\end{eqnarray}
\noindent
The usual elementary proof of (\ref{wallis12}) presented in textbooks
is to produce a recurrence
for $J_{2,m}$. Writing
$\cos^{2} \theta = 1 - \sin^{2} \theta$
and using integration by parts yields
\begin{eqnarray}
J_{2,m} & = & \frac{2m-1}{2m} J_{2,m-1}. \label{recur1}
\end{eqnarray}
\noindent
Now verify that the right side of (\ref{wallis12}) satisfies the 
same recursion and that both sides give $\pi/2$ for $m=0$. 

A second elementary proof of  Wallis's formula, also given in
 \cite{sarah1}, is done using a simple {\em double-angle trick}:
\begin{eqnarray}
J_{2,m} & = & \int_{0}^{\pi/2} \cos^{2m} \theta \, d \theta  = 
      \int_{0}^{\pi/2} \left( \frac{1  + \cos 2 \theta}{2} \right)^{m} 
\, d \theta.  \nonumber
\end{eqnarray}
\noindent
Now introduce the change of variables $\psi = 2 \theta$, expand and simplify the result by 
observing that the odd powers of cosine integrate to zero. Hence (\ref{wallis12}) is reduced 
to an inductive proof of the binomial recurrence
\begin{eqnarray}
J_{2,m} & = & 2^{-m} \sum_{i=0}^{\lfloor{ m/2 \rfloor}} 
\binom{m}{2i} J_{2,i}. \label{recur}
\end{eqnarray}
\noindent
Note that $J_{2,m}$ is uniquely determined by (\ref{recur}) along with the
initial value $J_{2,0} = \pi/2$. Thus  (\ref{wallis12}) now follows from the
identity 
\begin{eqnarray}
f(m) := \sum_{i=0}^{\lfloor{ m/2 \rfloor}} 
2^{-2i} \binom{m}{2i} \binom{2i}{i} & = &
2^{-m} \binom{2m}{m} \label{sum1}
\end{eqnarray}
\noindent
since (\ref{sum1}) can be written as 
\begin{eqnarray}
J_{2,m} & = & 2^{-m} \sum_{i=0}^{\lfloor{m/2\rfloor}} 
\binom{m}{2i} J_{2,i}, \nonumber
\end{eqnarray}
\noindent
where
\begin{eqnarray}
J_{2,i} & = & \frac{\pi}{2^{2i+1}} \binom{2i}{i}. \nonumber
\end{eqnarray}

The last step is to verify the identity (\ref{sum1}). This 
can be done {\em mechanically} using the theory 
developed by Wilf and Zeilberger, which is explained in
\cite{nemes,aequalsb}.  The sum in (\ref{sum1}) is the example used 
in \cite{aequalsb} (page 113) to 
illustrate their method. 

\smallskip

\noindent
{\bf Note}. The WZ-method is an algorithm  in Computational Algebra that, among 
other things, will produce for a hypergeometric/holonomic 
sum, such as (\ref{bin-sum}), a 
recurrence like (\ref{rec-22}). The reader will find in \cite{nemes} 
and \cite{aequalsb} information about this algorithm.  \\

\noindent The command
$$ct(binomial(m,2i) \, binomial(2i,i) 2^{-2i}, 1, i, m,N)$$
produces
\begin{eqnarray}
f(m+1) & = & \frac{2m+1}{m+1} \; f(m), \label{recur2}
\end{eqnarray}
\noindent
a recursion satisfied by the sum.  One completes the proof by
verifying that $2^{-m} \binom{2m}{m}$ satisfies the same 
recursion. Note that (\ref{recur1}) and (\ref{recur2}) are 
equivalent since $J_{2,m}$ and $f(m)$ differ only by a factor of 
$\pi/2^{m+1}$.

We have seen that Wallis's formula can be proven by an angle-doubling trick followed by a hypergeometric sum 
evaluation.  Perhaps the most interesting application of the double-angle trick
 is in the theory
of rational Landen transformations. See \cite{manna-moll3} for 
an overview. 

Now we employ the same ideas in the
evaluation of (\ref{int-qua}). The 
change of variables $x = \tan \theta$ 
yields 
\begin{equation}
N_{0,4}(a;m) = \int_{0}^{\pi/2} \left( \frac{\cos^{4} \theta}
{\sin^{4}\theta + 2a \sin^{2}\theta \cos^{2}\theta + \cos^{4}\theta} 
\right)^{m+1} \times 
\frac{d \theta}{\cos^{2} \theta}. \nonumber
\end{equation}
\noindent
Observe first that the denominator of the trigonometric function in the 
integrand is a polynomial in $u = 2 \theta$. In detail, 
\begin{equation}
\sin^{4}\theta + 2a \sin^{2}\theta \cos^{2}\theta + \cos^{4}\theta 
 = 2 \left[ (1+a) + (1-a) \cos^{2} u \right].
\nonumber
\end{equation}
\noindent 
In terms of the double-angle $u = 2 \theta$, the original integral becomes
\begin{equation}
N_{0,4}(a;m) = 2^{-(m+1)} \int_{0}^{\pi} 
\left( \frac{(1+ \cos u)^{2}}{(1+a) + (1-a) \cos^{2}u } \right)^{m+1} 
\times \frac{du}{1+ \cos u}. \nonumber
\end{equation}
\noindent
Next, expand the binomial $(1+ \cos u)^{2m+1}$ and check that 
\begin{equation}
\int_{0}^{\pi} \left[ (1 + a) + (1-a) \cos^{2}u \right]^{-(m+1)} 
\, \cos^{j} u \, du = 0
\label{vanishing}
\end{equation}
\noindent
for $j$ odd. The vanishing of half of the terms in the binomial expansion 
turns out to be a crucial property. The remaining integrals, those with 
$j$ even, can be simplified by using the double-angle trick one more time. 
The result is 
\begin{equation}
N_{0,4}(a;m)  =   \sum_{j=0}^{m} 2^{-j} \binom{2m+1}{2j} 
\int_{0}^{\pi} \left[ (3+a) + (1-a) \cos v \right]^{-(m+1)} ( 1 + \cos v)^{j} 
\, dv, \nonumber 
\end{equation}
\noindent
where $v = 2u$ and we have used the symmetry of cosine about $v = \pi$ to 
reduce the integrals form $[0, 2 \pi]$ to $[0, \pi]$. The  familiar 
change of variables $z = \tan(v/2)$ produces (\ref{int-qua}) with 
the complicated formula
\begin{equation}
d_{l,m} = \sum_{j=0}^{l} \sum_{s=0}^{m-l} \sum_{k=s+l}^{m} 
\frac{(-1)^{k-l-s}}{2^{3k}} 
\binom{2k}{k} \binom{2m+1}{2s+2j} \binom{m-s-j}{m-k} \binom{s+j}{j} 
\binom{k-s-j}{l-j}. \nonumber 
\end{equation}

\medskip

\noindent
{\bf Note}. In spite of its complexity, obtaining this expression was
the first step in the mathematical road described in this paper. 
It was precisely what Kauers and Paule \cite{kauers-paule} required to clarify 
some combinatorial properties of $d_{l,m}$. Some arithmetical properties 
can be read directly from it. For 
example, we can see that $d_{l,m}$ is a rational number and that 
$2^{3m} d_{l,m} \in \mathbb{Z}$; that is, its denominator is a power of $2$
bounded above by $3m$. Improvements on this bound are outlined in
Section \ref{sec-single}.

\section{A single sum expression for $d_{l,m}$} \label{sec-single} 
\setcounter{equation}{0}

The idea of doubling the angle that proved productive in Section 
\ref{sec-triple} can be expressed in the realm of rational functions via the 
change of variables 
\begin{equation}
y = R_{2}(x) := \frac{x^{2}-1}{2x}. 
\label{transf-r2}
\end{equation}
\noindent
The inverse has two branches 
\begin{equation}
x = y \pm \sqrt{y^{2}+1},
\end{equation}
\noindent
where the plus sign is valid for $x \in (0, \, +\infty)$ and the other one 
on $(-\infty, 0)$. The rational function $R_{2}$ arises from the 
identity
\begin{equation}
\cot 2 \theta = R_{2}(\cot \theta). 
\end{equation}
\noindent
This change of variables gives the proof of the next theorem.

\begin{Thm}
\label{thm-inv}
Let $f$ be a rational function and assume that the integral of $f$ over 
$\mathbb{R}$ is finite. Then 
\begin{eqnarray}
\int_{-\infty}^{\infty} f(x) \, dx & =  & 
\int_{-\infty}^{\infty}  \left[ f(y + \sqrt{y^{2}+1}) + f(y - \sqrt{y^{2}+1}) 
\right] \, dy +  \label{invar} \\
& + & \int_{-\infty}^{\infty} 
\left[ f(y + \sqrt{y^{2}+1}) - f(y - \sqrt{y^{2}+1}) 
\right] \, \frac{y \, dy}{\sqrt{y^{2}+1}}.
\nonumber
\end{eqnarray}
\noindent
Moreover, if $f$ is an 
{\em even} rational function,  the identity (\ref{invar}) remains valid if 
one replaces each interval of integration by  ${\mathbb{R}}^{+}$.
\end{Thm}

\begin{Thm}
For $m \in \mathbb{N}$,  let 
\begin{equation}
Q(x) = \frac{1}{(x^{4}+2ax^{2}+1)^{m+1}}. 
\end{equation}
\noindent
Define
\begin{eqnarray}
Q_{1}(y) & := & \left[ Q(y+\sqrt{y^{2}+1}) + Q(y-\sqrt{y^{2}+1}) 
\right] +  \nonumber \\
& + & \frac{y}{\sqrt{y^{2}+1}} 
\left[ Q(y+\sqrt{y^{2}+1}) - Q(y-\sqrt{y^{2}+1}) 
\right].   \nonumber
\end{eqnarray}
\noindent
Then 
\begin{equation}
Q_{1}(y) =  \frac{T_{m}(2y)}{2^{m}(1+a+2y^{2})^{m+1}},
\end{equation}
\noindent
where 
\begin{equation}
T_{m}(y) = \sum_{k=0}^{m} \binom{m+k}{m-k} y^{2k}. 
\label{bin-sum}
\end{equation}
\end{Thm}
\begin{proof}
Introduce the variable  $\phi = y + \sqrt{y^{2}+1}$. Then $y - \sqrt{y^{2}+1} = 
- \phi^{-1}$ and  $y = \tfrac{1}{2}(\phi - \phi^{-1})$.  Moreover, 
\begin{eqnarray}
Q_{1}(y) & = & \left[ Q(\phi) + Q(\phi^{-1}) \right] + 
\frac{\phi^{2}-1}{\phi^{2}+1} 
\left( Q(\phi) - Q(\phi^{-1}) \right) \nonumber \\
 & = & \frac{2}{\phi^{2}+1} \left[ \phi^{2}Q(\phi) + Q(\phi^{-1}) 
\right] \nonumber \\
& := & S_{m}(\phi). \nonumber 
\end{eqnarray}
\noindent
The result of the theorem is therefore equivalent to
\begin{equation}
2^{m} \left(1 + a + \tfrac{1}{2}(\phi - \phi^{-1})^{2} ) \right)^{m+1} \, 
S_{m}(\phi) = T_{m}(\phi - \phi^{-1}). 
\label{newform}
\end{equation}
\noindent
A direct simplification of the left hand side of (\ref{newform}) shows that 
this identity is equivalent to proving
\begin{equation}
\frac{\phi^{2m+1}+ \phi^{-(2m+1)}}{\phi+\phi^{-1}} = T_{m}(\phi - \phi^{-1}). 
\label{newform2}
\noindent
\end{equation}

To establish this, one simply 
checks that both sides of (\ref{newform2}) satisfy the 
second order recurrence 
\begin{equation}
c_{m+2}-  ( \phi^{2} + \phi^{-2}) c_{m+1} + c_{m} = 0, 
\label{rec-22}
\end{equation}
\noindent
and the values for $m=0$ and $m=1$ match. This is straight-forward for the 
expression on the left hand side, while the WZ-method 
settles the right hand side. 
\end{proof}

\bigskip

We now prove (\ref{int-qua}). 
The identity in Theorem \ref{thm-inv} shows that 
\begin{equation}
\int_{0}^{\infty}  Q(x) \, dx  =  \int_{0}^{\infty} Q_{1}(y) \, dy, 
\label{equalint}
\end{equation}
\noindent
and this last integral can be evaluated in elementary terms. Indeed, 
\begin{eqnarray}
\int_{0}^{\infty}  Q_{1}(y) \, dy & = & \int_{0}^{\infty} 
\frac{T_{m}(2y) \, dy}{2^{m} 
(1+ 2y^{2})^{m+1}} \nonumber \\
& = & \frac{1}{2^{m}} \sum_{k=0}^{m} \binom{m+k}{m-k} 
\int_{0}^{\infty} \frac{(2y)^{2k} \, dy}{(1 + a + 2y^{2})^{m+1}}. \nonumber
\end{eqnarray}
\noindent
The change of variables $y = t \, \sqrt{1+a}/\sqrt{2}$ gives
\begin{equation}
\int_{0}^{\infty}  Q_{1}(y) \, dy = \frac{1}{[2(1+a)]^{m+1/2}} 
\sum_{k=0}^{m} \binom{m+k}{m-k} 2^{k} (1+a)^{k} 
\int_{0}^{\infty}  \frac{t^{2k} \, dt}{(1+t^{2})^{m+1}}, \nonumber 
\end{equation}
\noindent
and the elementary identity
\begin{equation}
\int_{0}^{\infty} \frac{t^{2k} \, dt}{(1+t^{2})^{m+1}} = \frac{\pi}{2^{2m+1}} 
\binom{2k}{k} \binom{2m-2k}{m-k} \binom{m}{k}^{-1}
\nonumber
\end{equation}
gives
\begin{equation}
\int_{0}^{\infty} Q_{1}(y) \, dy 
= \frac{\pi}{2^{2m+1}} \frac{1}{[2(1+a)]^{m+1/2}} 
\sum_{k=0}^{m} \binom{m+k}{m-k} 2^{k} 
\binom{2k}{k} \binom{2m-2k}{m-k} \binom{m}{k}^{-1}
(1+a)^{k}.
\nonumber
\end{equation}
\noindent
This can be simplified further using 
\begin{equation}
\binom{m+k}{m-k} \binom{2k}{k} = \binom{m+k}{m} \binom{m}{k}
\end{equation}
\noindent 
and the equality (\ref{equalint}) to produce 
\begin{equation}
\int_{0}^{\infty} Q(y) \, dy = \frac{\pi}{2^{2m+1}} \frac{1}{[2(1+a)]^{m+1/2}} 
\sum_{k=0}^{m} 2^{k} \binom{m+k}{m}  \binom{2m-2k}{m-k} (1+a)^{k}. 
\end{equation}
This completes the proof of (\ref{int-qua}).  The 
coefficients $d_{l,m}$ are given by 
\begin{equation}
d_{l,m} = 2^{-2m} \sum_{k=l}^{m} 2^{k} \binom{2m-2k}{m-k} \binom{m+k}{m} 
\binom{k}{l}.
\label{positive}
\end{equation}
\noindent
This is clearly an improvement over the expression for $d_{l,m}$ given in the 
previous section. 

We now see that $d_{l,m}$ is a {\em positive} rational number. The 
bound on the  denominator is now improved to $2m-1$. This comes directly from 
(\ref{positive}) and the familiar fact 
that the  central binomial coefficients $\binom{2m}{m}$ are even. 

\section{A finite sum} \label{sec-finite} 
\setcounter{equation}{0}

The previous two sections have provided two expressions for the 
polynomial $P_{m}(a)$. The elementary evaluation in Section \ref{sec-triple}
gives 
\begin{eqnarray}
P_{m}(a) & = & \sum_{j = 0}^{m} \binom{2m+1}{2j} (a + 1)^{j} 
\sum_{k=0}^{m - j} \binom{m - j}{k} \binom{2(m-k)}{m-k} 2^{-3(m-k)} 
(a - 1)^{m - k - j} \nonumber \\
 & &   \label{poly1}
\end{eqnarray}
\noindent
and the results described in Section \ref{sec-single} provide the alternative 
expression
\begin{eqnarray}
P_{m}(a) & = & 2^{-m} \sum_{k=0}^{m} 
           2^{-k} \binom{2k}{k} \binom{2m-k}{m} (a+1)^{m-k}.  \label{poly2} \\
\nonumber
\end{eqnarray}
\noindent
The reader will find details in \cite{sarah1}. Comparing 
the values at $a=1$ given by both expressions leads to 
\begin{equation}
\sum_{k=0}^{m} 2^{-2k} \binom{2k}{k} \binom{2m+1}{2k} = 
\sum_{k=0}^{m} 2^{-2k} \binom{2k}{k} \binom{2m-k}{m}. 
\label{pretty} 
\end{equation}

\noindent
The identity (\ref{pretty}) can be verified using D. Zeilberger's package
EKHAD \cite{aequalsb}. Indeed, EKHAD tells us that both sides of 
(\ref{pretty}) satisfy the recursion 
\begin{equation}
(2m+3)(2m+2)f(m+1) =  (4m+5)(4m+3) f(m).  \nonumber
\end{equation}
\noindent
To conclude the proof by recursion, we check that they agree at $m=1$.  
A symbolic evaluation of both sides of (\ref{pretty}) leads to 
\begin{equation}
\frac{2^{2m+1} \Gamma(2m+3/2)}{\sqrt{\pi} \, \Gamma(2m+2)} = 
- \frac{2^{2m+1} \, \sqrt{\pi}}{\Gamma(-2m-1/2) \Gamma(2m+2)}. 
\end{equation}
\noindent
The identity (\ref{pretty}) now follows from 
\begin{equation}
\Gamma( m + \tfrac{1}{2} ) = \frac{\sqrt{\pi}}{2^{2m}} \frac{(2m)!}{m!} \text{ for } m \in \mathbb{N}.
\end{equation}

An elementary proof 
of (\ref{pretty}) would be desirable. \\

The left hand sum admits a 
combinatorial interpretation: multiply by $2^{2m+1}$ to produce 
\begin{equation}
S_{1}(m) := \sum_{j=0}^{m} \binom{2m+1}{2j} \binom{2j}{j} 2^{2m+1-2j}. 
\end{equation}
\noindent 
Consider the set $X$ of all paths in the plane that start at $(0,0)$ and take $2m+1$
steps in any of the four compass directions ($N = (0,1), \, S = (0,-1), \, E = (1,0)$ and $W= (-1,0)$) 
so that the path ends on the $y$-axis. Clearly there must be the same 
number of $E's$ and $W's$, say $j$ of them. Then to produce one of these 
paths, choose which is $E$ and which is $W$ in $\binom{2j}{j}$ ways. 
Finally, choose the remaining $2m+1-2j$ steps to be either $N$ or $S$, in 
$2^{2m+1-2j}$ ways. This shows that the set $X$ has $S_{1}(m)$ elements. 

Now let $Y$ be the set of all paths of the $x$-axis that start and end at $0$, 
take steps $e = 1$ and $w = -1$, and have length $4m+2$. The cardinality of 
$Y$ is clearly $\binom{4m+2}{2m+1}$. There is a simple bijection between the 
sets $X$ and $Y$ given by $E \to ee, \, W \to ww, N \to ew, \, S \to we$. 
Therefore, 
\begin{equation}
S_{1}(m) = \binom{4m+2}{2m+1}.
\end{equation}

We have been unable to produce a combinatorial proof for the right hand side of 
(\ref{pretty}). 

\section{A related family of polynomials} \label{sec-related}
\setcounter{equation}{0}

The expression (\ref{positive}) provides an efficient formula for the 
evaluation of $d_{l,m}$ when $l$ is close to $m$. For example, 
\begin{equation}
d_{m,m} = 2^{-m} \binom{2m}{m}  \text{ and } d_{m-1,m} = (2m+1)2^{-(m+1)} 
\binom{2m}{m}. 
\end{equation}
\noindent
Our attempt to produce 
a similar formula for small $l$ led us into a surprising family of 
polynomials. \\

The original idea is very simple: start with 
\begin{equation}
P_{m}(a) = \frac{2}{\pi} \left[ 2(a+1) \right]^{m + \tfrac{1}{2}} 
\int_{0}^{\infty} \frac{dx}{(x^{4} + 2ax^{2} + 1)^{m+1}}, 
\end{equation}
\noindent
and compute $d_{l,m}$ as coming from the Taylor expansion at $a=0$ of the 
right hand side. This yields 
\begin{equation}
d_{l,m} = \frac{1}{l!m!2^{m+l}} 
\left( \alpha_{l}(m) \prod_{k=1}^{m} (4k-1) - 
\beta_{l}(m) \prod_{k=1}^{m} (4k+1) \right),
\label{dlm-long}
\end{equation}
\noindent
where $\alpha_{l}$ and $\beta_{l}$ are polynomial in $m$ of degrees $l$ and 
$l-1$, respectively. The explicit expressions 
\begin{equation}
\alpha_{l}(m) = \sum_{t=0}^{\lfloor{ l/2 \rfloor} }
\binom{l}{2t} \prod_{\nu=m+1}^{m+t} (4 \nu -1) 
\prod_{\nu=m-l+2t+1}^{m} (2 \nu +1) 
\prod_{\nu=1}^{t-1} (4 \nu +1), 
\label{alpha}
\end{equation}
\noindent
and 
\begin{equation}
\beta_{l}(m) = \sum_{t=1}^{\lfloor{ (l+1)/2 \rfloor} }
\binom{l}{2t-1} \prod_{\nu=m+1}^{m+t-1} (4 \nu +1) 
\prod_{\nu=m-l+2t}^{m} (2 \nu +1) 
\prod_{\nu=1}^{t-1} (4 \nu -1),
\label{beta}
\end{equation}
\noindent
are given in \cite{bomosha}. \\

Trying to obtain more information about 
$\alpha_{l}$ and $\beta_{l}$ directly from (\ref{alpha}, \ref{beta}) proved 
difficult. One uninspired day, we decided to compute their roots
numerically. We were pleasantly surprised to discover the following property.

\begin{theorem}
For all $l \geq 1$, all the roots of $\alpha_{l}(m) = 0$ lie on the line 
$\realpart{m} = - \tfrac{1}{2}$. Similarly, the roots of 
$\beta_{l}(m) = 0$ for $l \geq 2$ lie on the same vertical line.
\end{theorem}

The proof of this theorem,  due to J. Little \cite{little}, starts by writing
\begin{equation}
A_{l}(s) := \alpha_{l}( (s-1)/2) \text{ and } 
B_{l}(s) := \beta_{l}( (s-1)/2)
\end{equation}
\noindent
and proving that $A_{l}$ is equal to $l!$ times the coefficient of $u^{l}$ in
$f(s,u) g(s,u)$, where $f(s,u) = (1+ 2u)^{s/2}$ and 
$g(s,u)$ is the hypergeometric series 
\begin{equation}
g(s,u) = {_{2}F_{1}} \left( \frac{s}{2}+ \frac{1}{4}, \frac{1}{4}; \frac{1}{2}; 
4 u^{2} \right).
\end{equation}
\noindent
A similar expression is obtained for $B_{l}(s)$. From here it follows that 
$A_{l}$ and $B_{l}$ each satisfy the  three-term recurrence 
\begin{equation}
x_{l+1}(s) = 2sx_{l}(s) - (s^{2} - (2l-1)^{2})x_{l-1}(s). 
\end{equation}
\noindent
Little then establishes a version of Sturm's theorem to prove the final result.  \\

The location of the zeros of $\alpha_{l}(m)$ now suggest to study the 
behavior of this family as $l \to \infty$. In the best of all worlds, one 
will obtain an analytic function of $m$ with all the zeros on a vertical 
line. Perhaps some Number Theory will enter and ... {\em one never knows}.

\section{Arithmetical properties} \label{sec-arith}
\setcounter{equation}{0}

The expression (\ref{dlm-long}) gives
\begin{equation}
m! 2^{m+1} \, d_{1,m} = (2m+1) \prod_{k=1}^{m} (4k-1) - \prod_{k=1}^{m} (4k+1),
\label{d1}
\end{equation}
\noindent
from where it follows that the right hand side is an even number. This led 
naturally to the problem of determining the $2$-adic valuation of 
\begin{eqnarray}
A_{l,m} := l! m! 2^{m+l} d_{l,m} & = & 
\alpha_{l}(m) \prod_{k=1}^{m} (4k-1) - \beta_{l}(m) \prod_{k=1}^{m} (4k+1) 
\label{new-A} \\
& = & \frac{l! m!}{2^{m-l}} \sum_{k=l}^{m} 2^{k} 
\binom{2m-2k}{m-k} \binom{m+k}{k} \binom{k}{l}.
\label{new1-A}
\end{eqnarray}

Recall that, for $x \in \mathbb{N}$, the $2$-adic valuation $\nu_{2}(x)$ is the
highest power of $2$ that divides $x$. This is extended to $x = a/b \in 
\mathbb{Q}$ via $\nu_{2}(x) = \nu_{2}(a) - \nu_{2}(b)$, leaving $\nu_{2}(0)$
as undefined. It follows from (\ref{new1-A}) that 
\begin{equation}
A_{m,m} = 2^{m} (2m)! \text{ and } A_{m-1,m} = 2^{m-1} (2m-1)! (2m+1),
\end{equation}
\noindent
so these $2$-adic valuations can be computed directly from Legendre's 
classical formula
\begin{equation}
\nu_{2}(x) = x - s_{2}(x),
\end{equation}
\noindent
where $s_{2}(x)$ counts the number of $1$'s in the  binary expansion of $x$. 

At the other end of the $l$-axis, 
\begin{equation}
A_{0,m} = \prod_{k=1}^{m} (4k-1)
\end{equation}
\noindent
is clearly odd, so $\nu_{2}(A_{0,m}) = 0$. The first interesting case is $l=1$:
\begin{equation}
A_{1,m} = (2m+1) \prod_{k=1}^{m} (4k-1) - \prod_{k=1}^{m} (4k+1).
\label{d1-new}
\end{equation}

The main result of \cite{bomosha} is that 
\begin{equation}
\nu_{2}(A_{l,m}) = \nu_{2}(m(m+1)) + 1.
\end{equation}

This was extended in \cite{amm1}. 

\begin{theorem}
\label{2adicall}
The $2$-adic valuation of $A_{l,m}$ satisfies 
\begin{equation}
\nu_{2}(A_{l,m}) = \nu_{2}((m+1-l)_{2l}) + l, 
\label{2valuel}
\end{equation}
\noindent
where $(a)_{k} = a(a+1) \cdots (a+k-1)$ is the Pochhammer symbol for 
$k \geq 1$. For $k=0$, we define $(a)_{0}=1$.
\end{theorem}

The proof is an elementary application of the WZ-method. Define the numbers
\begin{eqnarray}
B_{l,m} & := & \frac{A_{l,m}}{2^{l} (m+1-l)_{2l}},
\end{eqnarray}
\noindent
and use the WZ-method to obtain the recurrence 
\begin{equation}
B_{l-1,m}  =  (2m+1)B_{l,m} -(m-l)(m+l+1)B_{l+1,m}, \quad 1 \leq l \leq 
m-1. \nonumber
\end{equation}
\noindent
Since the initial values $B_{m,m} = 1$ and $B_{m-1,m} = 2m+1$ are odd, it 
follows inductively that $B_{l,m}$ is an odd integer. The reader will also 
find in \cite{amm1}
a WZ-free proof of the theorem.  \\

\noindent
{\bf Note}. The reader 
will find in \cite{amm2} a study of the $2$-adic valuation 
of the Stirling numbers. This study was motivated by the results described 
in this section. The papers \cite{cohen1, cohn1, lengyel1, lengyel2, wannemacker1, wannemacker2} contain information
about $2$-adic valuations of related sequences.  \\

\section{The combinatorics of the valuations} \label{sec-combina1}
\setcounter{equation}{0}

The sequence of valuations $\{ \nu_{2}(A_{l,m}): \, m \geq l \}$ increase in
complexity with $l$. Some of the combinatorial nature of this sequence is 
described next. The first feature of this sequence is that it has a block
structure, reminiscent of the simple functions of Real Analysis. 

\begin{definition}
Let $s \in \mathbb{N}, \, s \geq 2$. We say that a sequence 
$\{ a_{j}: \, j \in \mathbb{N} \}$ has {\em block structure} if there is an
$s \in \mathbb{N}$ such that
each $t \in \{ 0, \, 1, \, 2, \, \cdots \}$, we have
\begin{equation}
a_{st + 1} = a_{st + 2} = \cdots = a_{s(t+1)}.
\label{repeat}
\end{equation}
The sequence is called {\it $s$-simple} if $s$ is the largest value for which
(\ref{repeat}) occurs.
\end{definition}

\begin{theorem}
\label{period-thm}
For each $l \geq 1$, the set 
$X(l) := \{ \nu_{2}(A_{l,m}): \, m \geq l \, \}$ is an
$s$-simple sequence, with $s = 2^{1+ \nu_{2}(l)}$. 
\end{theorem}

We now provide a combinatorial interpretation for $X(l)$. This requires the 
maps 
\begin{eqnarray}
F ( \{ a_{1}, \, a_{2}, \, a_{3}, \, \cdots \} )   & := & 
\{ a_{1}, \, a_{1}, \, a_{2}, \, a_{3}, \, \cdots \}, \nonumber \\
T ( \{ a_{1}, \, a_{2}, \, a_{3}, \, \cdots \} )  & := &   
\{ a_{1}, \, a_{3}, \, a_{5}, \, a_{7}, \, \cdots \}. \nonumber
\end{eqnarray}
\noindent
We will also employ the notation
$c  :=  \{ \nu_{2}(m): \, m \geq 1 \}
= \{ 0, \, 1, \, 0, \, 2, \, 0, \, 1, \cdots \}$.

\medskip

We describe an algorithm that reduces the sequence $X(l)$ 
to a constant sequence.
The algorithm starts with the sequence 
$X(l)  :=  \left\{ \nu_{2}( A_{l,l+m-1}): \quad m \geq 1 \, \right\}$ and 
then finds and $n \in \mathbb{N}$ so that $X(l)$ is 
is $2^{n}$-simple. Define 
$Y(l)  :=  T^{n} \left( X(l) \right)$. At the initial stage, 
Theorem \ref{period-thm}
ensures that $n= 1 + \nu_{2}(l)$. The next step is to introduce the shift
$Z(l) :=   Y(l) - c$ and finally define $W(l):= F(Z(l))$. 
If $W(l)$ is a constant sequence, then STOP; otherwise repeat the 
process with $W$ instead of $X$.  Define $X_k(l)$ as the new 
sequence at the end of the $(k-1)$th cycle of this process, with
$X_1(l)=X(l)$. 

This algorithm produces a sequence of integers $n_{j}$, so that 
$X_{k}(l)$ is $2^{n_{k}}$-simple. The integer vector
$\Omega(l)  :=  \left\{ n_{1}, \, n_{2}, \, n_{3}, \cdots, n_{\omega(l)} 
\right\} $ is called the {\em reduction sequence} of $l$. The number 
$\omega_{l}$ is the number of cycles requires to obtain a constant sequence.

\begin{definition}
\label{def-compo}
Let $l \in \mathbb{N}$. The {\em composition} of $l$, denoted 
by $\Omega_{1}(l)$, is an integer sequence
defined as follows: Write $l$ in binary form. 
Read the digits from right to left. The first part of $\Omega_{1}(l)$ is 
the number of digits up to and including the first $1$ read in the 
corresponding binary sequence; the second one is the number of additional 
digits up to and including the second $1$ read, and so on until the number
has been read completely.
\end{definition}

\begin{theorem}
\label{thm-reduc}
Let $\{ k_{1}, \cdots, k_{n}: \, 0 \leq k_{1} < k_{2} < \cdots < k_{n} \}$, be
the unique collection of distinct nonnegative integers such that
$ l = \sum_{i=1}^{n} 2^{k_{i}}$. Then the 
reduction sequence $\Omega(l)$ 
of $l$ is $\{ k_{1}+1, \, k_{2}-k_{1}, \cdots, 
k_{n}-k_{n-1} \}$. 
\end{theorem}

It follows that the reduction 
sequence $\Omega(l)$ is precisely the sequence of compositions of $l$, that is,
$\Omega(l) = \Omega_{1}(l)$. This is the combinatorial interpretation of the 
algorithm used to reduce $X(l)$ to a constant sequence.

\section{Valuation patterns encoded in binary trees} \label{sec-tree}
\setcounter{equation}{0}

In this section we describe the precise structure of the graph 
of the sequence $\{ \nu_{2}(A_{l,m}), m \geq l \}$. The reader is referred to \cite{moll-sun} for complete 
details. In view of the block structure
described in the previous section, it suffices to consider the sequences
$\{ \nu_{2}(C_{l,m}), m \geq l \}$,  which are defined by $$ C_{l,m} =  $$   
The emerging patterns are still very complicated. For instance, Figure 
\ref{val-c13} shows the case $l=13$ and Figure \ref{val-c59} 
corresponds to $l=59$. The remarkable
fact is that in spite of the complexity of $\nu_{2}(C_{l,m})$
there is {\em an exact formula} for it. The rest of this section describes 
how to find it. 
{{
\begin{figure}[h]
\begin{center}
\centerline{\psfig{file=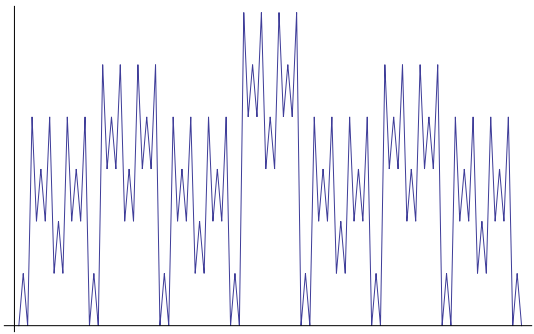,width=20em,angle=0}}
\caption{The valuation $\nu_{2}(C_{13,m})$}
\label{val-c13}
\end{center}
\end{figure}
}}
{{
\begin{figure}[h]
\begin{center}
\centerline{\psfig{file=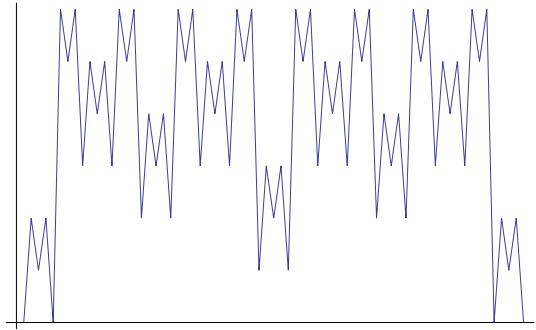,width=20em,angle=0}}
\caption{The valuation $\nu_{2}(C_{59,m})$}
\label{val-c59}
\end{center}
\end{figure}
}}

We describe now the {\em decision tree} associated to the index $l$. 
Start with a root $v_{0}$ at level $k=0$. To 
this vertex we attach the 
sequence $\{ \nu_{2}(C_{l,m}): m \geq 1 \}$ and ask whether 
$\nu_{2}(C_{l,m})-\nu_{2}(m)$ has a constant value 
{\em independent} of $m$. If the 
answer is yes, we 
say that $v_{0}$ is a {\em terminal vertex} and label it with 
this constant. The tree is complete. If the answer is negative, we 
split the integers modulo $2$ and produce two new vertices, $v_{1}, \, v_{2}$,
connected to $v_{0}$ and 
attach to the classes $\{ \nu_{2}(C_{l,2m-1}): m \geq 1 \}$ and 
$\{ \nu_{2}(C_{l,2m}): m \geq 1 \}$ to these vertices. We now ask whether 
$\nu_{2}(C_{l,2m-1})-\nu_{2}(m)$ is independent of $m$ and the same for 
$\nu_{2}(C_{l,2m})-\nu_{2}(m)$.  Each vertex that yields a positive answer is 
considered terminal and the corresponding constant value is attached to it. 
Every vertex with a negative answer produces two new ones at the next level. 

Assume the vertex $v$ corresponding to the sequence 
$\{ 2^{k}(m-1) + a: \, m \geq 1 \}$ produces a negative answer. Then it 
splits in the next generation into two vertices corresponding to the 
sequences $\{ 2^{k+1}(m-1) + a: \, m \geq 1 \}$ and 
$\{ 2^{k+1}(m-1) + 2^{k} + a: \, m \geq 1 \}$. For 
instance, in Figure \ref{tree5}, the
vertex corresponding to $\{ 4m: \, m \geq 1 \}$, that is not terminal, splits 
into $\{ 8m: \, m \geq 1 \}$ and 
$\{ 8m - 4: \, m \geq 1 \}$. These two edges lead to terminal vertices. 
Theorem \ref{formula-val2} shows that this 
process ends in a finite number of steps.  \\

{{
\begin{figure}[ht]
\begin{center}
{\centerline{\psfig{file=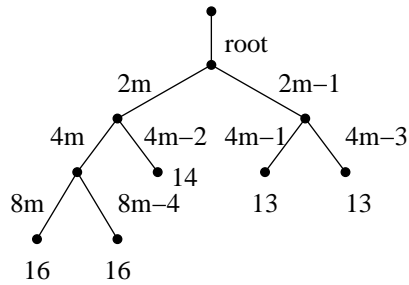,width=15em,angle=0} }}
\caption{The decision tree for $l=5$}
\label{tree5}
\end{center}
\end{figure}
}}

\begin{theorem}
\label{formula-val2}
Let $l \in \mathbb{N}$ and $T(l)$ be its decision tree. Define 
$k^{*}(l) := \lfloor{ \log_{2}l \rfloor}$. Then  \\

\noindent
1) $T(l)$ depends only on the odd part of $l$; that is, for $r \in \mathbb{N}$, 
we have $T(l) = T(2^{r}l)$, up to the labels.  \\

\noindent
2) The generations of the tree are labelled starting at $0$; that is, the root 
is generation $0$. Then, for $0 \leq k \leq k^{*}(l)$, the $k$-th generation 
of $T(l)$ has $2^{k}$ vertices. Up to that point, $T(l)$ is a complete 
binary tree.  \\

\noindent
3) The $k^{*}$-th generation contains $2^{k^{*}+1}-l$ terminal vertices. The 
constants associated with these vertices are given by the following algorithm. 
Define 
\begin{equation}
j_{1}(l,k,a) := -l + 2(1+2^{k}-a), 
\end{equation}
\noindent
and 
\begin{equation}
\gamma_{1}(l,k,a) = l+k+1 + \nu_{2} \left( (j_{1}+l-1)! \right) + 
\nu_{2} \left( (l-j_{1})! \right). 
\end{equation}
\noindent
Then, for $1 \leq a \leq 2^{k^{*}+1}-l$, we have
\begin{equation}
\nu_{2} \left( C_{l,2^{k}(m-1)+a} \right) = \nu_{2}(m) + 
\gamma_{1}(l,k,a). 
\end{equation}
\noindent
Thus, the vertices at the $k^{*}$-th generation have constants given by 
$\gamma_{1}(l,k,a)$.  \\

\noindent
4) The remaining terminal vertices of the tree $T(l)$ appear in the 
next generation. There are $2(l-2^{k^{*}(l)})$ of them.  The constants 
attached to these vertices are defined as follows: let
\begin{equation}
j_{2}(l,k,a) := -l + 2(1+2^{k+1}-a), 
\end{equation}
\noindent
and 
\begin{equation}
j_{3}(l,k,a) := j_{2}(l,k,a+2^{k}). 
\end{equation}
\noindent
Define 
\begin{equation}
\gamma_{2}(l,k,a) := l+k+2 + \nu_{2} \left( (j_{2} + l -1)! \right)
+ \nu_{2} \left( (l- j_{2})! \right), 
\end{equation}
\noindent
and 
\begin{equation}
\gamma_{3}(l,k,a) := l+k+2 + \nu_{2} \left( (j_{3} + l -1)! \right)
+ \nu_{2} \left( (l- j_{3})! \right).
\end{equation}
\noindent
Then, for $2^{k^{*}(l)+1}-l+1 \leq a \leq 2^{k^{*}(l)}$, we have 
\begin{equation}
\nu_{2} \left( C_{l,2^{k^{*}(l) +1}(m-1)+a} \right) = 
\nu_{2}(m) + \gamma_{2}(l,k^{*}(l),a), 
\end{equation}
\noindent
and 
\begin{equation}
\nu_{2} \left( C_{l,2^{k^{*}(l) +1}(m-1)+a+2^{k^{*}(l)} } \right) = 
\nu_{2}(m) + \gamma_{3}(l,k^{*}(l),a), 
\end{equation}
\noindent
give the constants attached to these remaining terminal vertices. 
\end{theorem}

\medskip

We now use the theorem to produce a formula for $\nu_{2}(C_{3,m})$. 
The value $k^{*}(3) = 1$ shows that the first level contains 
$2^{1+1}-3 = 1$ terminal vertex. This corresponds to the sequence 
$2m-1$ and has constant value $7$, thus, 
\begin{equation}
\nu_{2} \left(C_{3,2m-1} \right) = 7. 
\end{equation}
\noindent
The next level has $2(3 - 2^{1}) = 2$ terminal vertices. These correspond to 
the sequences $4m$ and $4m-2$, with constant values $9$ for both of them.
This tree produces 
\begin{equation}
\nu_{2} \left( C_{3,m} \right) = \begin{cases} 
7 + \nu_{2} \left( \tfrac{m+1}{2} \right) & \text{ if }  m \equiv 1 \bmod 2, \\
9 + \nu_{2} \left( \tfrac{m}{4} \right) & \text{ if }  m \equiv 0 \bmod 4, \\
9 + \nu_{2} \left( \tfrac{m+2}{4} \right) & \text{ if }  m \equiv 2 \bmod 4.
\end{cases}
\end{equation}

The complexity of the graph for $l=13$ is reflected in the  analytic formula
for this valuation. The theorem yields

\begin{equation}
\nu_{2} \left( C_{13,m} \right) = \begin{cases} 
36 + \nu_{2} \left( \tfrac{m+7}{8} \right) & \text{ if }  m \equiv 1 \bmod 8, \\
37 + \nu_{2} \left( \tfrac{m+6}{8} \right) & \text{ if }  m \equiv 2 \bmod 8, \\
36 + \nu_{2} \left( \tfrac{m+5}{8} \right) & \text{ if }  m \equiv 3 \bmod 8, \\
40 + \nu_{2} \left( \tfrac{m+12}{16} \right) & \text{ if }  m \equiv 4 \bmod 16, \\
38 + \nu_{2} \left( \tfrac{m+11}{16} \right) & \text{ if }  m \equiv 5 \bmod 16, \\
39 + \nu_{2} \left( \tfrac{m+10}{16} \right) & \text{ if }  m \equiv 6 \bmod 16, \\
38 + \nu_{2} \left( \tfrac{m+9}{16} \right) & \text{ if }  m \equiv 7 \bmod 16, \\
40 + \nu_{2} \left( \tfrac{m+8}{16} \right) & \text{ if }  m \equiv 8 \bmod 16, \\
40 + \nu_{2} \left( \tfrac{m+4}{16} \right) & \text{ if }  m \equiv 12 \bmod 16, \\
38 + \nu_{2} \left( \tfrac{m+3}{16} \right) & \text{ if }  m \equiv 13 \bmod 16, \\
39 + \nu_{2} \left( \tfrac{m+2}{16} \right) & \text{ if }  m \equiv 14 \bmod 16, \\
38 + \nu_{2} \left( \tfrac{m+1}{16} \right) & \text{ if }  m \equiv 15 \bmod 16, \\
40 + \nu_{2} \left( \tfrac{m}{16} \right) & \text{ if }  m \equiv 16 \bmod 16. 
\end{cases}
\end{equation}

The details for Theorem \ref{formula-val2} are given in \cite{moll-sun}. \\

\noindent
{\bf Note}. The $p$-adic valuations of $A_{l,m}$ for $p$ odd present 
phenomena different from those explained for the case $p=2$. Figure \ref{val-17}
shows the plot of $\nu_{17}(A_{1,m})$ where we observe linear growth. 
Experimental data suggest that, for any odd prime $p$, one has 
\begin{equation}
\nu_{p}(A_{l,m}) \sim \frac{m}{p-1}. 
\end{equation}
\noindent
Figure \ref{error-17}
 depicts the error term $\nu_{17}(A_{1,m}) - m/16$. The structure
of the error remains to be explored. 

{{
\begin{figure}[h]
\begin{center}
\centerline{\psfig{file=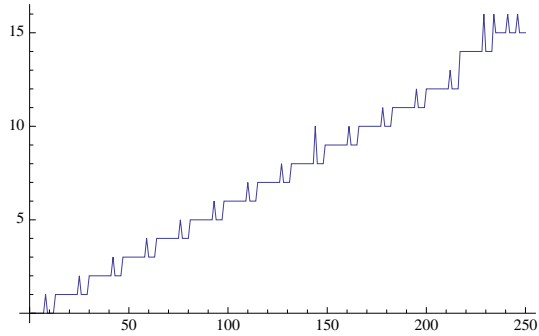,width=20em,angle=0}}
\caption{The valuation $\nu_{17}(A_{1,m})$}
\label{val-17}
\end{center}
\end{figure}
}}

{{
\begin{figure}[h]
\begin{center}
\centerline{\psfig{file=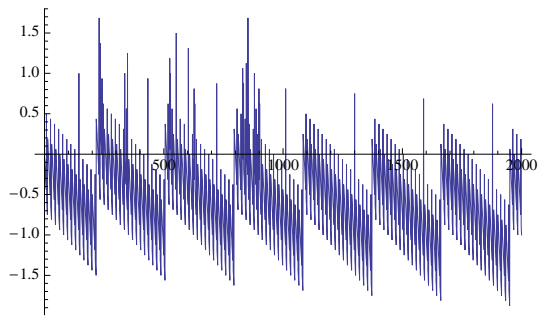,width=20em,angle=0}}
\caption{The error term $\nu_{17}(A_{1,m}) - m/16$}
\label{error-17}
\end{center}
\end{figure}
}}

\section{Unimodality and log-concavity} \label{sec-unimodal}
\setcounter{equation}{0}

A finite sequence of real 
numbers $\{ a_{0}, \, a_{1}, \cdots, a_{m} \}$ is said to 
be {\em unimodal} if there exists an index $0 \leq j \leq m$ such that 
$a_{0} \leq a_{1} \leq \cdots \leq a_{j}$ and 
$a_{j} \geq a_{j+1} \geq \cdots \geq a_{m}$. A polynomial is said to 
be unimodal 
if its sequence of coefficients is unimodal.  The sequence 
$\{a_{0}, a_{1}, \cdots, a_{m} \}$ with $a_{j} \geq 0$ is said
to be {\it logarithmically concave} (or {\em log-concave} for short) if 
$a_{j+1}a_{j-1} \leq a_{j}^{2}$ for $ 1 \leq j \leq m-1$. It is easy to see
that if a sequence is log-concave then it is unimodal \cite{wilf1}. 

Unimodal polynomials arise often in combinatorics, geometry, and algebra, and 
have been the subject of considerable research in recent years. The reader is
referred to \cite{stanley1} and \cite{brenti1} for surveys of the 
diverse techniques employed to 
prove that specific families of polynomials are unimodal. 

For $m \in \mathbb{N}$, the sequence $\{ d_{l,m}: 0 \leq l \leq m \}$ is 
unimodal. This is a consequence of the following criterion established in 
\cite{bomouni1}. 

\begin{theorem}
Let $a_{k}$ be a nondecreasing sequence of positive numbers and let 
$A(x) = \sum_{k=0}^{m} a_{k} x^{k}$. Then $A(x+1)$ is unimodal. 
\end{theorem}

We applied this theorem to the polynomial 
\begin{equation}
A(x)  := 2^{-2m} \sum_{k=0}^{m} 2^{k} \binom{2m-2k}{m-k} 
\binom{m+k}{m} x^{k}
\label{niceA}
\end{equation}
\noindent
that satisfies  $P_{m}(x) = A(x+1)$.  The criterion was extended in \cite{bomouni2} to include the shifts $A(x+j)$
and in \cite{wang2} for arbitrary shifts. The original proof of the unimodality 
of $P_{m}(a)$ can be found in \cite{bomouni3}.  \\

In \cite{moll-notices} we conjectured the log-concavity of
$\{ d_{l,m}: 0 \leq l \leq m \}$. This turned out 
a more  difficult question. Here we describe some of our failed attempts.  \\

\noindent
1) A result of Brenti \cite{brenti1} states that if 
$A(x)$ is log-concave then so is $A(x+1)$. Unfortunately this does not apply in 
our case since (\ref{niceA}) is not log-concave. Indeed, 
\begin{eqnarray}
2^{4m-2k} \left( a_{k}^{2} - a_{k-1}a_{k+1} \right) & = & 
\binom{2m}{m-k}^{2} \binom{m+k}{m}^{2} \times  \nonumber \\
 & \times & 
\left( 1 - \frac{k(m-k)(2m-2k+1)(m+k+1)}{(k+1)(m+k)(2m-2k-1)(m-k+1)} \right)
\nonumber
\end{eqnarray}
\noindent
and this last factor could be negative---for example, for $m=5$ and $j=4$. 
The number of negative terms in this sequence is small, so perhaps there is 
a way out of this. \\

\noindent
2) The coefficients $d_{l,m}$ satisfy many recurrences. For example,
\begin{equation}
d_{j+1}(m)  =  \frac{2m+1}{j+1} d_{j}(m) - 
\frac{(m+j)(m+1-j)}{j(j+1)} d_{j-1}(m). 
\end{equation}
This can be found by a direct application of WZ method. Therefore, $d_{l,m}$ 
is logconcave provided 
\begin{equation}
j(2m+1) d_{j-1}(m) d_{j}(m) \leq  (m+j)(m+1-j) d_{j-1}(m)^{2} + 
j(j+1) d_{j}(m)^{2}.
\end{equation}
\noindent
We have conjectured that the smallest value of the expression 
\begin{equation}
(m+j)(m+1-j)d_{j-1}(m)^{2} + j(j+1)d_{j}(m)^{2} - j(2m+1) d_{j-1}(m)d_{j}(m)
\end{equation}
\noindent
is $2^{2m} m(m+1) \binom{2m}{m}^{2}$ and it occurs at $j=m$. This would imply
the log-concavity of $\{ d_{l,m} : 0 \leq l \leq m \}$.  Unfortunately, it 
has not yet been proven. \\

Actually we have conjectured that the $d_{l,m}$ satisfy a stronger version of log-concavity. Given 
a sequence $ \{ a_{j} \}$  of positive numbers, define a map 
\begin{equation}
\mathfrak{L} \left( \{ a_{j} \} \right)  :=  \{ b_{j} \} \nn
\end{equation}
\no
by $b_{j}  :=   a_{j}^{2} - a_{j-1}a_{j+1}$. Thus 
$\{ a_{j} \}$ is log-concave if  $\{ b_{j} \}$  has 
positive coefficients.  The nonnegative sequence $\{ a_{j} \}$ is 
called 
{\em infinitely log-concave} if any number of applications of $\mathfrak{L}$ 
produces a nonnegative sequence.  \\

\begin{conjecture}
\label{conj-inf}
For each fixed $m \in \mathbb{N}$, the 
sequence $\{ d_{l,m} : 0 \leq l \leq m \}$ is infinitely log-concave.
\end{conjecture}

\medskip

The log-concavity of  $\{ d_{l,m} : 0 \leq l \leq m \}$ has recently been established by M. Kauers 
and P. Paule \cite{kauers-paule} as an applications of their work on 
establishing inequalities by automatic means. The starting point is the
triple sum expression in Section \ref{sec-triple} written as
\begin{equation}
d_{l,m} = \sum_{j,s,k} \frac{(-1)^{k+j-l}}{2^{3(k+s)}} 
\binom{2m+1}{2s} \binom{m-s}{k} \binom{2(k+s)}{k+s} \binom{s}{j} 
\binom{k}{l-j}. \nonumber
\end{equation}
\noindent
Using the RISC package Multisum \cite{wegschaider1} they derive the 
recurrence
\begin{equation}
2(m+1)d_{l,m+1} = 2(l+m)d_{l-1,m} + (2l+4m+3)d_{l,m}. 
\label{recu1}
\end{equation}
\noindent
The positivity of $d_{l,m}$ follows directly from here. To establish the 
log-concavity of $d_{l,m}$ the new recurrence
\begin{equation}
4l(l+1)d_{l+1,m} = 
-2(2l-4m-3)(l+m+1)d_{l,m} + 4(l-m-1)(m+1)d_{l,m+1}
\nonumber 
\end{equation}
\noindent
is derived automatically and the log-concavity of $d_{l,m}$ is reduced to 
establishing the inequality
\begin{eqnarray} 
d_{l,m}^{2} & \geq & 
\frac{4(m+1) \left( 4((l-m-1)(m+1)-(2l^{2}-4m^{2}-7m-3)d_{l,m+1}d_{l,m} \right)}
{16m^{3}+16lm^{2}+40m^{2}+28lm + 33m +9l+9}
\nonumber
\end{eqnarray}
\noindent
The $2$-log-concavity of  $\{ d_{l,m} : 0 \leq l \leq m \}$, that is ${\mathfrak{L}}^{(2)}( \{ d_{l,m}\})
\geq \{0,0, \dots 0\}$ remains an open question. At the end of \cite{kauers-paule} the 
authors state that
 ``...we have little hope that a proof of $2$-logconcavity could be completed
along these lines, not to mention that a human reader would have a hard time 
digesting it." \\

The general concept of infinite log-concavity has generated some interest. D. Uminsky 
and K. Yeats \cite{umi-yeats} have studied the action of  $\mathfrak{L}$ on 
sequences of the form 
\begin{equation}
\{\cdots, 0, 0, 1, x_{0}, x_{1}, \cdots, x_{n}, 
\cdots, x_{1}, x_{0}, 1,0,0, \cdots \}
\end{equation}
\noindent
and 
\begin{equation}
\{\cdots, 0, 0, 1, x_{0}, x_{1}, \cdots, x_{n}, x_{n}, 
\cdots, x_{1}, x_{0}, 1,0,0, \cdots \}
\end{equation}
\noindent
and established the existence of a large unbounded region in the positive 
orthant of $\mathbb{R}^{n}$ that consists only of infinitely log-concave
sequences $\{x_0, \dots, x_n\}$. P. McNamara and B. Sagan \cite{mcnamara1} have considered 
sequences satisfying the condition $a_{k}^{2} \geq r a_{k-1}a_{k+1}$. Clearly
this implies log-concavity of $r \geq 1$. Their techniques apply to the 
rows of the Pascal triangle. Choosing appropriate 
$r$-factors and a computer verification procedure, they obtain the following.

\begin{theorem}
For fixed $n \leq 1450$, the sequence $ \{ \binom{n}{k}: 0 \leq k \leq n\}$ is 
infinite log-concave. 
\end{theorem}

\noindent
In particular, they looked for values of $r$ for which the $r$-factor 
condition is preserved
by the $\mathfrak{L}$-operator.  The factor that works is 
$r = \frac{3 + \sqrt{5}}{2}$ (the square of the golden mean).  This 
technique can be used on a variety of finite sequences.  
See \cite{mcnamara1} for a complete discussion of the techinque.

McNamara and Sagan have also considered 
$q$-analogues of the binomial coefficients. 
In order to describe these extensions we introduce the basic notation 
and refer to the reader to \cite{kac-chung} and \cite{andrews3} for more
details on the world of $q$-analogues. 

Let $q$ be a variable and for $n \in \mathbb{N}$ define 
\begin{equation}
[n]_{q} = \frac{1-q^{n}}{1-q} = 1 + q + q^{2} + \cdots + q^{n-1}.
\end{equation}
\noindent
The {\em Gaussian-polynomial} or {\em $q$-binomial coefficients} are defined
by 
\begin{equation} 
\begin{bmatrix} n \\ k \end{bmatrix}_{q}   = 
\frac{ [n]_{q}!}{[k]_{q}! \, [n-k]_{q}!},
\end{equation}
\noindent
where $[n]_{q}! := [1]_{q} [2]_{q} \cdots [n]_{q}$. The Gaussian polynomials 
have nonnegative coefficients. We will say that the sequence of polynomials 
$ \{ f_{k}(q) \}$ is {\it $q$-log-concave} if $\mathfrak{L}( f_{k}(q) )$ is a 
sequence of polynomials with nonnegative coefficients. The extension of
this definition to {\it infinite q-log-concavity} is made in the obvious way.

Observe that 
\begin{equation}
\binom{n}{k} = \lim\limits_{q \to 1} 
\begin{bmatrix} n \\ k \end{bmatrix}_{q}. 
\end{equation}
\noindent
McNamara and Sagan have established the surprising result:

\begin{theorem}
The sequence 
$ {\left \{ \begin{bmatrix} n \\ k \end{bmatrix}_{q} \right\}}_{k \geq 0}$ is
not infinite q-logconcave. 
\end{theorem}

In fact they established that applying $\mathfrak{L}$ twice gives polyomials
with some negative coefficients. As a compensation, they propose: 

\begin{conjecture}
The sequence 
$ \left \{ \begin{bmatrix} n \\ k \end{bmatrix}_{q} \right\}_{n \geq k}$ is
infinite q-log-concave for all fixed $k \geq 0$. 
\end{conjecture}

Another $q$-analog of the binomial coefficients that arisies in the study of 
quantum groups is defined by 
\begin{equation}
\langle{ n \rangle}  := \frac{q^{n}-q^{-n}}{q-q^{-1}} = 
\frac{1}{q^{n-1}}( 1 + q^{2}+q^{4} + \cdots + q^{2n-2}).
\end{equation}
\noindent
From here we proceed as in the case of Gaussian polynomials and define
\begin{equation}
\quant := \frac{\quann !}{\quank ! \, \quannk !}
\end{equation}
\noindent
where $\quann! = \quanone \langle{2 \rangle} \cdots \langle{n \rangle}$.  For 
these coefficients 
McNamara and Sagan have proposed

\begin{conjecture}
\noindent
a) The row sequence $ \left\{ \quant \right\}_{k \geq 0}$ is infinitely 
$q$-log-concave for all $n \geq 0$. 

\noindent
b) The column sequence $ \left\{ \quant \right\}_{n \geq k}$ is infinitely 
$q$-log-concave for all fixed $n \geq 0$. 

\noindent
c) For all integers $0 \leq u < v$, the sequence 
$ \left\{ \quanuv \right\}_{m \geq 0}$ is infinitely 
$q$-log-concave for all $n \geq 0$. 
\end{conjecture}

This conjecture has been verified for all $n \leq 24$ with $v \leq 10$. When 
$u>v$, using 
\begin{equation}
\quant = \frac{1}{q^{nk-k^{2}}} \begin{bmatrix} n \\ k \end{bmatrix}_{q^{2}}
\end{equation}
\noindent
one checks that the lowest degree of 
$\quanthree^{2} - \quantfour$ is $-1$, so the sequence is not even 
$q$-log-concave. Sagan and McNamara observe that when 
$u=v$, the quantum groups analoge has exactly the same 
behavior as the Gaussian polynomials. 

Newton began the study of log-concave sequences by establishing the following
result (paraphrased in Section $2.2$ of \cite{hardy7}). 

\begin{theorem}
Let $\{ a_{k} \}$ be a finite sequence  of positive real numbers. Assume all 
the roots of the polynomial 
\begin{equation}
P[a_{k};x]:= a_{0} + a_{1}x + \cdots + a_{n}x^{n}
\end{equation}
\noindent
are real. Then the sequence $\{ a_{k} \}$ is log-concave. 
\end{theorem}

McNamara and Sagan \cite{mcnamara1} and, independently, R. Stanley 
have proposed the next conjecture.

\begin{conjecture}
Let $\{ a_{k} \}$ be a finite sequence of positive real numbers. If 
$P[a_{k};x]$ has only real roots then the  same is true for 
$P[ \mathfrak{L}(a_{k}); x]$. 
\end{conjecture}

\medskip

This conjecture was also independently made by Fisk.  See \cite{mcnamara1} 
for the complete details on the conjecture. \\

The polynomials 
$P_{m}(a)$ in (\ref{polyP-def}) are the generating function for the 
sequence $\{ d_{l,m} \}$ described here. It is an unfortunate fact that 
they do not have real roots
\cite{bomouni3} so these conjecture would not imply Conjecture 
\ref{conj-inf}. In spite of this, the asymptotic behavior of these zeros 
has remarkable properties. Dimitrov \cite{dimitrov1} has shown that, in the 
right scale, the zeros converge to a lemniscate. \\

The infinite-log-concavity of $\{ d_{l,m} \}$ has resisted all our efforts. It 
remains to be established. \\

\medskip

\noindent
{\bf Acknowledgements}. The authors wish to thank B. Sagan for many comments on 
an earlier version of the paper. 
The first author acknowledges the partial support of 
NSF-DMS 0713836.

\end{document}